\documentclass[11pt]{amsart}
\usepackage{amsmath, amssymb, latexsym}
\usepackage{hyperref}

\usepackage{xcolor}

\newtheorem{theorem}{Theorem}[section]
\newtheorem{lemma}[theorem]{Lemma}
\newtheorem{proposition}[theorem]{Proposition}
\newtheorem{remark}[theorem]{Remark}

\newtheorem{definition}{Definition}[section]
\newtheorem{corollary}[theorem]{Corollary}

\newcommand{\beit}{\begin{itemize}}

\newcommand{\cent}{\begin{center}}

\newcommand{\eit}{\end{itemize}}
\newcommand{\eeq}{\end{equation}}
\newcommand{\eeqa}{\end{eqnarray*}}
\newcommand{\ent}{\end{center}}

\newcommand{\seq}{\begin{equation}}
\newcommand{\seqa}{\begin{eqnarray*}}

\newcommand{\barint}{-\kern-.175in\displaystyle\int}

\newcommand{\norm}[1]{\mbox{$\left\| #1 \right\|$}}

\def\XXint#1#2#3{{\setbox0=\hbox{$#1{#2#3}{\int}$}
    \vcenter{\hbox{$#2#3$}}\kern-.5\wd0}}

\begin{document}
\title[Weak weighted estimates]{Restricted weighted weak boundedness  for product TYPE  operators}

 \author {Mar\'{\i}a Jes\'us Carro \and   Sheldy Ombrosi}

\address{Mar\'ia J. Carro,   Department of  Analysis and Applied Mathematics,              
Universidad Complutense de Madrid, Plaza de las Ciencias 3,
28040 Madrid, Spain.}
\email{mjcarro@ucm.es}

\address{Sheldy Ombrosi,   Department of  Analysis and Applied Mathematics,              
Universidad Complutense de Madrid, Plaza de las Ciencias 3,
28040 Madrid, Spain.}
\email{sombrosi@ucm.es}

 \subjclass[2010]{Primary: 47B90  Secondary: 42B15}

\keywords{Muckenhoupt weights, Hilbert transform, Fourier multiplier.}

\thanks{The authors were partially supported by  grants PID2020-113048GB-I00 funded by MCIN/AEI/10.13039/501100011033,  CEX2019-000904-S funded by MCIN/AEI/ 10.13039/501100011033 and Grupo UCM-970966 (Spain)}

\begin{abstract}  
Given a  bilinear (or sub-bilinear) operator $B$, we prove restricted weighted weak type inequalities of the form 
$$
||B(f_1, f_2)||_{L^{p, \infty}(w_1^{p/p_1}w_2^{p/p_2})}\lesssim ||f_1||_{L^{p_1, 1}(w_1)}||f_2||_{L^{p_2, 1}(w_2)}, 
$$
whenever $B(f_1, f_2)= (T_1f_1) (T_2 f_2)$ is the product of two singular integral operators satisfying  Dini Conditions. Additionally, we also establish, as an application, the boundedness of a certain class of bounded variation bilinear Fourier multipliers  solving a question posted in \cite{bcls}. 
 \end{abstract}

\maketitle
 
\section{Introduction}

Multilinear operators likely first appeared in the work of Coifman and Meyer  in the 1970s  \cite{cm1:cm1, cm2:cm2} when they were interested in the study of certain singular integral operators, such as the Calder\'on commutators, paraproducts, and pseudodifferential operators. Additionally, the solution of the Calder\'on's conjecture on the  boundedness of the bilinear Hilbert transform by Lacey and Thiele (\cite{lt:lt})  gave a strong motivation to Grafakos and Torres for the study of multilinear singular integrals (see \cite{gt:gt}). 

One of the simplest example of a bilinear (or sub-bilinear) operator is given by the product of two  linear (resp. sublinear) operators
$$
B(f_1, f_2)= (T_1f_1)(T_2 f_2), 
$$
whenever they are well defined. Weighted boundedness of the type 
\begin{equation}\label{easy}
||B(f_1, f_2)||_{L^p(w_1^{p/p_1}w_2^{p/p_2})}\lesssim ||f_1||_{L^{p_1}(w_1)}||f_2||_{L^{p_2}(w_2)}
\end{equation}
under the condition 
\begin{equation}\label{stad1}
\frac 1p=\frac 1{p_1}+\frac 1{p_2}, \qquad p_j\ge 1, 
\end{equation}
are trivial by H\"older's inequality assuming $T_j$ are bounded on $L^{p_j}(w_j)$. These inequalities have been widely studied mainly in the setting of Muckenhoupt weights ( see \cite{lw:lw, ft:ft, ft3:ft3,ft2:ft2,loptt:loptt, mn:mn}). Condition \eqref{stad1} is assumed all over the paper and we simply write $(p_1, p_2; p)$. 
 
Things change completely if we are interested in the so-called restricted weighted weak type inequalities
\begin{equation}\label{rwt}
||B(f_1, f_2)||_{L^{p, \infty}(w_1^{p/p_1}w_2^{p/p_2})}\lesssim ||f_1||_{L^{p_1, 1}(w_1)}||f_2||_{L^{p_2, 1}(w_2)}
\end{equation}
under the weaker condition that $T_j:L^{p_j,1}(w_j) \to L^{p_j, \infty}(w_j)$, $j=1,2$, whenever $w_1\neq w_2$. The main difficulty, to prove \eqref{rwt}, is the fact that weak Lorentz spaces do not satisfy H\"older's inequality; that is, 
$$
||f g||_{L^{p, \infty}(w_1^{p/p_1} w_2^{p/p_2})} \lesssim ||f||_{L^{p_1, \infty}(w_1)} ||g||_{L^{p_2, \infty}(w_2) }
$$
is false in general \cite{cr:cr}. Let us just mention that if $w_1=w_2=w$, then H\"older's inequality works perfectly and the restricted weak type inequality holds trivially.

In particular, if $T_1=T_2= M$ is the Hardy-Littlewood maximal operator defined for locally integrable functions on $\mathbb R^n$ by
$$
Mf(x):=\sup_{Q\ni x}\frac{1}{|Q|}\int_Q |f(y)|dy,
$$
where the supremum is taken over all cubes $Q\subseteq\mathbb R^n$ containing $x \in \mathbb R^n$, the restricted weighted weak type boundedness
of $B(f_1, f_2)=(Mf_1)(Mf_2)$ was proved in \cite{loptt:loptt} for the case $(1,1;1/2)$ and in \cite{pr:pr} for the rest of the cases.  Hence,  it is known that, for every $(p_1, p_2; p)$ with $p_1, p_2\ge 1$, and every $w_j\in A_{p_j}^{\mathcal R}$ (we refer to Section \ref{prel} for all the definitions and properties of the classes of weights that are going to appear in this introduction), 
\begin{equation}\label{max}
\|(Mf_1)(Mf_2)\|_{L^{p, \infty}(w_{1}^{p/p_1}w_{2}^{p/p_2})}\lesssim \|f_1\|_{L^{p,1}(w_1)}\|f_2\|_{L^{p,1}(w_2)}.
\end{equation}
 
 However, if $M$ is substituted, for example,  by the Hilbert transform 
$$
Hf(x)=p.v. \lim_{\varepsilon\to 0} \int_{|x-t|>\varepsilon} \frac{f(t)}{x-t} dt,
$$
that is, $B(f_1, f_2)=H_2(f_1, f_2):=(Hf_1) (Hf_2)$, the question was left open in \cite{bcls} and \cite{pr:pr}. One of the main results in this paper is to prove \eqref{rwt} not only for $H_2$ but for a much bigger class of operators including  the product of  Calder\'on-Zygmund operators satisfying a Dini condition assuming that $w_j\in A_{p_j}^\mathcal R$ $j=1,2$.   In fact, the class of operators,  we are going to work with,  will be those satisfying the called Condition (C) (see Definition \ref{condC}).  Let us just mention that sub-bilinear operator can be changed, with no extra effort, by sub-multilinear operator.

Moreover, a quantitative version of our results will provide, through Rubio de Francia’s extrapolation theory, an interesting application to a certain class of bounded variation bilinear Fourier multipliers. We would like to emphasize that more important than the specific example we are about to study is the technique we develop using restricted weighted weak-type inequalities, which enables us to solve an open problem discussed below.

Recall that given a bounded function $m(\xi, \eta)$ on $\mathbb R^n\times \mathbb R^n$,   a bilinear Fourier multiplier  is defined, initially  on pairs of Schwartz functions $f$ and $g$,  by the operator
$$
B_m (f, g)(x) := \int_{\mathbb R^{2n}} m(\xi, \eta) \hat f(\xi) \hat g(\eta) e^{i x\cdot(\xi+\eta)} d\xi d\eta, \qquad x \in \mathbb R^n.
$$
 
The theory of bilinear Fourier multipliers has been  strongly developed in the last years and  there are many results concerning weighted inequalities  of the form \eqref{easy}  (see \cite{gra:gra, gt:gt, ft:ft, ft3:ft3,ft2:ft2,loptt:loptt, mn:mn}).

\smallskip

Clearly, if $m(\xi, \mu)=m_1(\xi)m_2(\mu)$, then
$$
B_m(f_1, f_2)= (T_{m_1}f_1)(T_{m_2}f_2),
$$
with $T_{m_j}$ the standard linear Fourier multiplier; that is,
$$
T_{m_j} f_j(x)=\int_{\mathbb R^n} \Hat{f}_j(\xi) m_{j}(\xi) e^{2\pi x\cdot \xi}d\xi.
$$
In particular, if
$m(\xi, \mu)= \mbox{sgn} \,\xi \ \mbox{sgn}  \, \mu$, we have that
$
B_m =H_2. 
$

Let us now consider the following natural extension of  bounded variation functions on $\mathbb R$. Set
 \begin{equation}\label{motivation}
m(\xi, \eta) := \int_{-\infty}^{\xi}\int_{-\infty}^{\eta} d\mu(t, s),  \qquad \xi, \eta \in \mathbb R,
\end{equation}
where $d\mu$ is a finite measure on $\mathbb R^2$ that we can assume, without loss of generality, that $\mu(\mathbb R^2)=1$. This type of bilinear Fourier multiplier were introduced in \cite{r:r} and some weighted estimates were also given  in \cite{bcls}.

In fact, it is easy to see that, for nice functions $f$ and $g$, 
\begin{equation*}
B_m (f, g)(x)= \int_{\mathbb R^2} H_{t,s} (f, g)(x)  \, d\mu(t, s), \qquad  x \in \mathbb R,
\end{equation*}
where
\begin{equation*}\label{eq:Hst_operator}
    H_{t,s} (f, g)(x)= H_tf(x) H_s g(x), 
\end{equation*}
and
 $$
  H_rh(x) =\frac 12\left[ h(x) + i e^{2\pi i xt}H(e^{-2\pi i r\cdot} h)(x)\right] . 
    $$
 One of the original motivations of this paper was to prove an estimate that remained open in \cite{bcls}; namely
\begin{equation}\label{bmbm}
    B_m: L^1(u_1)\times L^1(u_2) \longrightarrow L^{1/2}(u_1^{1/2}u_2^{1/2}), 
 \end{equation}
    for every $u_1, u_2\in A_1$. The case $u_1=u_2$ was already solved there. To prove \eqref{bmbm},  using a Rubio de Francia extrapolation technique developed in \cite{r:r, cr1:cr1}, it will be enough to prove that, for some $(p_1, p_2; p)$ with $p>1$ and every $w_j\in \hat A_{p_j, 2}$, $j=1,2$, 
\begin{equation*}\label{porfin}
B_m:L^{p_1, 1}(w_1) \times L^{p_2, 1}(w_2) \longrightarrow L^{p, \infty}(w_1^{p/p_1} w_2^{p/p_2}), 
\end{equation*}
with a constant controlled by $\varphi([w_1]_{ \hat A_{p_1, 2}}, [w_2]_{ \hat A_{p_2, 2}})$ for some function $\varphi$ increasing in each variable.

The strategy, in both \cite{loptt:loptt}  and \cite{pr:pr} to obtain \eqref{max},   was to reduce the problem to a linear Sawyer mixed weak  estimate. Actually, the contribution here will be to show that a variant of that strategy can be applied to several operators beyond the Hardy-Littlewood maximal function and in particular to the Hilbert transform. As a consequence we shall obtain the following theorem. 

\begin{theorem} \label{Hilbert} For every $(p_1, p_2; p)$  and every  $w_j\in A_{p_j}^\mathcal R$ ($j=1,2$), 
$$
H_2:L^{p_1, 1}(w_1) \times L^{p_2, 1}(w_2) \longrightarrow L^{p, \infty}(w_1^{p/p_1} w_2^{p/p_2}).
$$
\end{theorem}

From here,  we shall be able to solve the open problem mentioned above.

\begin{corollary} \label{bv1} For every $m$ as in \eqref{motivation},  every $(p_1, p_2;p)$, every $q_1, q_2\ge 1$ so that $\frac 1{q_1} +\frac 1{q_2}<1$, and every $w_j\in A_{p_j}^\mathcal R$ ($j=1,2$), 
\begin{equation*}\label{porfin}
B_m:L^{p_1, \frac{p_1}{q_1}}(w_1) \times L^{p_2, \frac{p_2}{q_2}}(w_2) \longrightarrow L^{p, \infty}(w_1^{p/p_1} w_2^{p/p_2}). 
\end{equation*}
\end{corollary}

As a consequence,   \eqref{bmbm} holds: 

\begin{corollary} \label{bv2} For every $m$ as in \eqref{motivation},    and every $w_j\in   A_{1}$ ($j=0,1$), 
\begin{equation*}\label{porfin}
B_m:L^{1}(w_1) \times L^{1}(w_2) \longrightarrow L^{1/2, \infty}(w_1^{1/2} w_2^{1/2}). 
\end{equation*}
\end{corollary}

Now, to obtain our necessary Sawyer type inequality, we need to   prove the following result, which is an extension of the case $q=1$  obtained in \cite{CU-M-P}. 

\begin{theorem} \label{extra}
Given a family of functions \(\mathcal{F}\), suppose that for some  $0 < p_0 < \infty$  and every \(w \in A_\infty\),
\[
\int_{\mathbb{R}^n} f_1(x)^{p_0} w(x) \, dx \le \varphi([w]_{A_\infty}) \int_{\mathbb{R}^n} f_2(x)^{p_0} w(x) \, dx,
\]
for all \((f_1, f_2) \in \mathcal{F}\), such that the left-hand side is finite and   $\varphi$ is an increasing  function on $[0, \infty)$. Then, for every $q\ge 1$,  $u\in A_{\infty}$ and $v^q\in A_{\infty}$, 
$$
\| f_1 v^{-1} \|_{L^{q,\infty}(uv^q)} \le C(u, v)  \| f_2 v^{-1} \|_{L^{q,\infty}(uv^q)}, \quad \forall \, (f_1, f_2) \in \mathcal{F}. 
$$
Moreover, if $v=(Mh)^{-\alpha}$ with $h\in L^1_{loc}(\mathbb R^n)$ and $\alpha>0$ and  $u\in \widehat A_{q, 2}$,  then
\begin{equation}\label{buv}
C(u, v)\lesssim \phi([u]_{\widehat A_{q,2}}),
\end{equation}
with $\phi$ an increasing function. 
\end{theorem}
 
 \smallskip

Given an operator $T$,  set
$$
Sf=S_Tf:= M_{1/2}(Tf)(x)=(M(|Tf|^{1/2}(x))^{2}.
$$

 \begin{definition} \label{condC}
 We say that $T$  satisfies  condition $(C)$ if there is some  $0 < p_0 < \infty$, such that, for every $w \in A_\infty$, we have
$$
\int_{\mathbb{R}^n} Sf(x)^{p_0} w(x) \, dx \leq \varphi([w]_{A_\infty}) \int_{\mathbb{R}^n} Mf(x)^{p_0} w(x) \, dx,
$$
for every function $f$ such that the left hand side is finite and for some increasing function $\varphi$.
 \end{definition}
 
It is known that any Calder\'on-Zygmund with Dini condition meets condition $(C)$. This fact  was first observed in page 124 \cite{Al-Pe}, and for a more detailed proof see Section 3 in \cite{Or-Ca}. 

\smallskip

Now, in \cite{Li-O-P} it was proved that if $u\in A_1$ and $v\in A_{\infty}$ then  the inequality 
$$
\| Mfv^{-1} \|_{L^{1,\infty}(uv)} \leq C(u, v) \|f\|_{L^1(u)}
$$
holds and, in \cite{pr:pr},  a variant of that inequality was considered for others $q>1$; namely, for every $u\in A_q^{\mathcal R}$ ($q\ge 1$) and every $v$ such that $uv^q\in A_\infty$, 
\begin{equation}\label{mvmv}
\| Mfv^{-1} \|_{L^{q,\infty}(uv^q)} \leq C_M(u, v) \|f\|_{L^{q,1}(u)},
\end{equation}
where, if $uv^q\in A_r^{\mathcal R}$, 
\begin{equation}\label{vava}
C_M(u, v)=\Psi( [u]_{A_q^{\mathcal R}}, [uv^q]_{A_r^{\mathcal R}}),
\end{equation}
with $\Psi$ an increasing function in each variable,  depending only on $r$ and $q$. 

Therefore, that achievement combined with Theorem \ref{extra} produces the following result. Let us mention that the case $q=1$ was essentially  obtained in \cite{CU-M-P}.

\begin{theorem}\label{Saw}
Let $q\ge 1$  and  let $u$ and $v$ be weights such that   $v\in A_{\infty}$, $u\in A_q^{\mathcal R}$ and $u v^q \in A_\infty$. Then, if an operator  $T$  satisfies  condition $(C)$, we have that 

$$
\| Sf v^{-1} \|_{L^{q,\infty}(uv^q)} \leq C_S(u, v) \|f\|_{L^{q,1}(u)}. 
$$
Moreover, if $v=(Mh)^{-\alpha}$ for some $h\in L^1_{loc}$ and some $\alpha>0$ and $u\in \widehat A_{q, 2}$, then we can take 
$$
C_S(u, v)\sim \phi ([u]_{\widehat A_{q,2}}), 
$$
with $\phi$ an increasing function. 
 
\end{theorem}

Observe that since $|Tf(x)|\le Sf(x)$ a.e. $x$,  we also have that the same conclusion for the operator $T$. Now, using the previous result we can obtain our main result. 

\begin{theorem} \label {General}
Let $w_j\in A_{p_j}^{\mathcal R}$ ($j=0,1$), and let $T_j$ be two operators satisfying  condition $(C)$.  Then, for every $(p_1,p_2;p)$,  
$$
\|(T_1f_1)(T_2f_2)\|_{L^{p,\infty}(w_{1}^{p/p_1}w_{2}^{p/p_2})}\lesssim C(w_1, w_2) \|f_1\|_{L^{p_1, 1}(w_1)}\|f_2\|_{L^{p_2,1}(w_2)}.
$$

In addition, if  $w_j\in \widehat A_{p_j, 2}$ ($j=0,1$),  then
$$
C(w_1, w_2)\lesssim \varphi([w_1]_{\widehat A_{p_1, 2}}, [w_2]_{\widehat A_{p_2, 2}})  
$$
with $\varphi$ and increasing function in each variable. 
\end{theorem}

In the previous estimates, as well as those that follow, the symbols ``$\lesssim$'' and ``$\sim$'' indicate that the inequality or equivalence holds up to universal constants or depending on fixed parameters.

From this point, the paper is organized as follows. In the next section, we provide the necessary definitions and preliminaries. In Section 3, we prove the main results. Finally,  the applications (Corollaries 1.2 and 1.3) are provided in the last section.

\section{preliminaries}\label{prel}

Let us start by defining the classes of weights that have appeared in the introduction. For every $q> 1$, the Muckenhoupt weights $A_q$ \cite{Muckenhoupt TAMS72} are defined by the condition
$$
[w]_{A_q}:=\sup_{Q}\left(\frac{1}{|Q|}\int_Q w(x)dx\right)\left(\frac{1}{|Q|}\int_Q w(x)^{1-q'}dx\right)^{q-1}<\infty,
$$
where the supremum extends over all cubes $Q\subseteq \mathbb{R}^n$ and with $q'$ satisfying $1/q'+ 1/q = 1$. This class of weights characterizes the $L^q(w)$-boundedness of the Hardy-Littlewood maximal operator $M$. 

Moreover, if $u$ is a weight such that $Mu(x)\le C u(x)$ a.e $x\in\mathbb{R}^n$, $C > 0$, we say that $u\in A_1$, and denote by $[u]_{A_1}$ the least constant $C$ satisfying such inequality. In fact, for every $q \geq 1$, 
$$
M:L^q(w)\longrightarrow L^{q,\infty}(w) \qquad \Longleftrightarrow \qquad w\in A_q.
$$

For every $q \geq 1$, the restricted $A_q$ class, $A_q^{\mathcal{R}}$, is defined as the weights $w$ such that
$$
[w]_{A_q^{\mathcal{R}}}=\sup_{E\subseteq Q}\frac{|E|}{|Q|}\left(\frac{w(Q)}{w(E)}\right)^{1/q}<\infty,
$$
where the supremum is taken over all cubes $Q\subseteq\mathbb{R}^n$ and all measurable sets $E\subseteq Q$. It holds (see \cite{KT, chk:chk}) that
 \begin{equation*}\label{mhlrwt}
M:L^{q,1}(w) \longrightarrow L^{q,\infty}(w)\quad \iff \quad w \in A_q^{\mathcal{R}}.
\end{equation*}
 It is easy to see that, $A_q \subsetneq A_q^\mathcal R\subset A_{q+\varepsilon}$, for every $\varepsilon>0$ and, when $q = 1$, $A_1^{\mathcal{R}}=A_1$. 
Also we will consider the class of weights  $A_\infty$ given by
$$
A_\infty=\cup_{q>1}A_q, 
$$
and, it is easy to see that 
$$
A_\infty= \cup_{q>1} A_{q}^{\mathcal R}.
$$

 Moreover,  there are, in the literature,  many equivalent ways of defining $[w]_{A_\infty}$,  but the most typical one is the following Fujii-Wilson constant
$$
[w]_{A_\infty}= \sup_Q \frac 1{w(Q)}\int_Q M(w \chi_Q) dx. 
$$
It is also well known that an $A_{\infty}$ weight $w$ could be factorized as the product $w=uv$ with $u\in A_{1}$ and $v\in RH_{\infty}$,   
where $RH_\infty$ is the class of weights satisfying
$$
[v]_{RH_\infty}= \frac{(\sup_Q v(x))|Q|}   {v(Q)}<\infty. 
$$
See for instance Theorem 5.1 in \cite{CU-N}. It will be useful for us to consider the following constant associated with a weight $w\in A_{\infty}$, which we define as follows
$$
[w]_{RH_\infty^1}=\inf \{[u]_{A_1}[v]_{RH_\infty}: w=uv\}.
$$
 Also, the class $\widehat A_q$ is defined by
  $$
\widehat A_q=\{ w: \exists u\in A_1 \mbox{ and } h\in L^1_{loc}: w=u (Mh)^{1-q}\},
$$
with
$$
[w]_{\widehat A_q}=\inf_{u} [u]_{A_1}^{1/q}.
$$
In \cite{cgs:cgs},  it was proved that
$$
\widehat A_q\subset A_q^\mathcal R, \qquad [w]_{ A_q^{\mathcal R}}\lesssim [w]_{\widehat A_q}.
$$
For  our purposes, we also need to introduce the class $\widehat A_{q, 2}$ as the following set of weights: 
$$
\{w: \exists \, 0\le \alpha \le 1, u\in A_1, h_j\in L^1_{loc}: w=u (Mh_1)^{\alpha(1-q)}(Mh_2)^{(1-\alpha)(1-q)} \},
$$
with
$$
[w]_{\widehat A_{q,2}}=\inf_{u} [u]_{A_1}^{1/q},
$$
and can be easily proved by interpolation \cite{r:r} that  
\begin{equation*}\label{tec}
\widehat A_q\subset \widehat A_{q, 2}\subset A_q^\mathcal R, \qquad [w]_{ A_q^{\mathcal R}}\lesssim [w]_{\widehat A_{q,2}}\le  [w]_{\widehat A_{q}}.
\end{equation*}
In fact, since $v=(Mh_1)^{\alpha(1-q)}(Mh_2)^{(1-\alpha)(1-q)}\in RH_\infty$ with $[v]_{RH_\infty}\le C_{\alpha, q}$, we have that 
$$
\widehat A_{q, 2}  \subset  A_{\infty}, \qquad [w]_{RH_\infty^1} \lesssim [w]_{\widehat A_{q, 2}}^q.
$$
In general,   $w=u \prod_{j=1}^N (Mh_j)^{-\alpha_j}$,  with $u\in A_1$, $h_j\in L^1_{loc}$ and $\alpha_j\ge 0$, for every $j$ will be in $A_{\infty}$, and  $[w]_{RH_\infty^1}\lesssim [u]_{A_1}$.

All the above  classes of weights are strongly linked with the extrapolation theory of Rubio de Francia. In particular, we shall use the following result (see Theorem 4.2.6 in \cite{r:r}).

\begin{theorem}\label{th:restrictedGrafakos} 
Let $B$ be a bilinear operator.  Set $(p_1, p_2; p)$ with $p_j>1$ and assume that,  for every $v_j\in \widehat A_{p_j, 2}$,  
 $$
 B:L^{p_1,1}(w_1)\times L^{p_2,1}(w_2)\longrightarrow L^{p,\infty}(w_1^{p/p_1}w_2^{p/p_2})
 $$
 with constant less than or equal to 
   $\varphi \big ( [w_1]_{\widehat A_{p_1, 2}}, [w_2]_{\widehat A_{p_2, 2}} \big)$, where $\varphi$ is an increasing function in each coordinate. 
Then, for all indices $q_1,q_2 \geq 1$ such that  $q_1/p_1=q_2/p_2 \leq 1$, and for every weight $v_j\in \widehat{A}_{q_j}$,
\begin{equation}\label{nnn}
B:L^{q_1,\frac{q_1}{p_1}}(v_1)\times L^{q_2,\frac{q_2}{p_2}}(v_2) \longrightarrow L^{q,\infty}(v_1^{q/q_1}v_2^{q/q_2}),
\end{equation}
with constant less than or equal to $\phi \big([v_1]_{\hat A_{q_1}},  [v_2]_{\hat A_{q_2}}\big)$ where $\phi$ is an increasing function in each coordinate. 
\end{theorem}

\begin{remark} \label{mumu}  1) We observe that, for every $q>0$   condition \eqref{nnn} implies that, for every measurable sets $E$ and $F$, 
\begin{equation}\label{atomic}
||B(\chi_E, \chi_F)||_{L^{q,\infty}(v_1^{q/q_1}v_2^{q/q_2})}\lesssim v_1(E)^{1/q_1}v_2(F)^{1/q_2}. 
 \end{equation}
Therefore,  if  $q>1$,   \eqref{nnn} implies that 
$$
B:L^{q_1,1 }(v_1)\times L^{q_2,1}(v_2) \longrightarrow L^{q,\infty}(v_1^{q/q_1}v_2^{q/q_2}).
$$

\noindent
2) In general, if $q_1=q_2=1$, \eqref{atomic} is not equivalent to 
$$
B:L^{1 }(v_1)\times L^{1}(v_2) \longrightarrow L^{1/2,\infty}(v_1^{1/2}v_2^{1/2}).
$$
However, it was proved in \cite{bcls} that under certain hypotheses on $B$, namely to be  $(\varepsilon, \delta)$-atomic approximable in each variable,  these two conditions are equivalent. This is going to be the case in our application. 
\end{remark}

\begin{lemma}\label{infinito2} If  $w\in A_{\infty}$, then 
$$
[w]_{A_\infty}\lesssim [w]_{RH_\infty^1}^{2}. 
$$

\end{lemma}

\begin{proof} Let $w=uv$ with $u\in A_1$ and $v\in RH_\infty$. Then
\begin{eqnarray*}
[w]_{A_\infty} &\lesssim& \sup_Q \frac {\sup_{x\in Q} v }{w(Q)}  \int_Q Mu(x) dx \le [u]_{A_1}\sup_Q \frac {\sup_{x\in Q} v}{w(Q)} u(Q)
\\
&\le&  [u]_{A_1} [v]_{RH_\infty} \sup_Q \frac {v(Q) }{w(Q)} \frac{u(Q)}{|Q|}\le [u]_{A_1}^2 [v]_{RH_\infty}\le [u]_{A_1}^2 [v]_{RH_\infty}^2,
\end{eqnarray*}
and the result follows taking the infimum over all possible decomposition of $w=uv$. 
 \end{proof}

Now, given a weight $u_0\in A_1$, let us consider the operator
$$
L_{u_0}(x)=\frac {M(u_0 f)(x)}{u_0(x)} , \qquad a.e. x. 
$$
 
We will need the following proposition obtained in \cite{CU-M-P}.
\begin{proposition}  \label{qo} \cite{CU-M-P} Let $s\ge 1$.  Given $u_0\in A_1$ and $\nu$ such that  $\nu^{s}\in A_\infty$, 
 there exists $p_0>1$ and $K_0>0$ such that, for every $p\ge p_0$  
\begin{equation}\label{K0}
L_{u_0}: L^{p,1}(u_0\nu^s) \longrightarrow L^{p,1}(u_0\nu^s), \qquad ||L_{u_0}||\le K_0.
\end{equation}

Moreover, if  
$$
R_{u_0} h=\sum_{k=0}^\infty \frac{L_{u_0}^k h}{2^k K_0^k}, 
$$
 the following conditions holds:
\begin{enumerate}
\item
$h\le R_{u_0} h$.

\item $u_0(R_{u_0} h)\in A_1$ with $[u_{0}(R_{u_0} h)]_{A_1} \le 2K_0$.

\item 
$$
R_{u_0}: L^{p, 1} (u_0\nu^s) \longrightarrow L^{p,  1} (u_0\nu^s), \qquad [R_{u_0}] \le 2.
$$
 \end{enumerate}
\end{proposition}

For later purposes, we need to clarify a quantitative estimate of the above statement for particular choices of $\nu$, in order to apply Rubio de Francia's theorem in our application.

\begin{proposition} \label{later} In Proposition \ref{qo}, if we assume that $\nu=\prod_{j=1}^N(Mh_j)^{-\alpha_j}$ for some $\alpha_j>0$, then we can choose $p_0\sim  [u_0]_{A_1}$ and  $K_0\sim  [u_0]^2_{A_1}$.

\end{proposition}

\begin{proof} In the proof of Theorem 1.7 in \cite{CU-M-P} (or with some more details in the appendix of \cite{LiOPi}) an explicit expression for $K_0$ was obtained as follows 
$$
K_0=4 p_0 (C_0+C_1),
$$
   where $C_1=[u_0]_{A_1}$  and $C_0$ can be controlled by the norm $\norm{M}_{L^{p_0}(u_0^{1-p_0} \nu^s)}$. However,   by Buckey's Theorem (\cite{Buck}), the previous norm is controlled by $[u_0^{1-p_0} \nu^s]^{1/(p_0-1)}_{A_{p_0}}$.

 Now, in order to estimate this quantity, we write
 $$
 u_0^{1-p_0} \nu^s= \left[ u_0 \prod_{j=1}^N(Mh_j)^{\frac{\alpha_js}{p_0-1}}\right]^{1-p_0}. 
 $$
Let us see now that we can choose $p_0$ such that $u_1:= u_0 \prod_{j=1}^N(Mh_j)^{\frac{\alpha_js}{p_0-1}}\in A_1$   with $[u_1]_{A_1}\lesssim [u_0]_{A_1}$. Let $\varepsilon=\frac 1{2^{n+1}[u_0]_{A_1}}$. Then, it is known (see for instance Lemma 3.1 in \cite{LeOP})  that $u_2:=u_0^{1+\varepsilon} \in A_1$ with $[u_2]_{A_1}\lesssim [u_0]_{A_1}$.

Let us choose $p_0$ so that if $\beta_{j}=\frac{1+\varepsilon} {\varepsilon}\frac{\alpha_js}{p_0-1}$ then $\sum_{j=1}^{N}\beta_j=\frac{1}{2}$. Then

$$
u_1= u_2^{\frac 1{1+\varepsilon}} \Big(\prod_j (Mh_j)^{\beta_j}\Big)^{\frac{\varepsilon}{1+\varepsilon}}, 
$$ 
  and, consequently, 
  $$
  [u_1]_{A_1} \lesssim [u_2]_{A_1}\lesssim [u_0]_{A_1},
  $$
  as we wanted to see.  From here, it follows that
  $$
 [u_0^{1-p_0} \nu^s]^{1/(p_0-1)}_{A_{p_0}} \lesssim  [u_1]_{A_1}   \lesssim  [u_0]_{A_1}. 
   $$
  Since, clearly, $p_0\sim  [u_0]_{A_1}$,    the result follows.

\end{proof}

\section{Proofs of Theorems}

\begin{proof}[Proof of Theorem \ref{extra}] Using Kolmogorov inequalities, we know that, for every $r<q$, 
\begin{equation}\label{kolmo}
\left|\left| \frac {f_1} v\right|\right|_{L^{q, \infty}(uv^q)}\approx \left( \sup_E \frac 1{uv^q(E)^{1-\frac rq}} \int_E f_1^r u v^{q-r} \right)^{1/r}. 
\end{equation}

Since $u\in A_{\infty}$, there exists $u_0\in A_1$ and $v_0\in RH_\infty$ so that $u=u_0 v_0$.   Now, in the following estimate we will apply Proposition \ref{qo} with $u_0$, $s=q$ and $\nu= v_0^{1/q} v$. And therefore $u_{0}\nu^{q}= u v^{q}$. In fact, we have

\begin{eqnarray*}
& &\int_E f_1^r u v^{q-r}= \int f_1^r u_0 v_0 v^{q-r} \chi_E 
\le \int  f_1^r u_0 R_{u_0}(\chi_E) v_0 v^{q-r} 
\\
&\le& \varphi([u_0 R_{u_0}(\chi_E) v_0v^{q-r}]_{A_\infty})  \int  f_2^r u_0 R_{u_0}(\chi_E) v_0 v^{q-r},
\end{eqnarray*}

Observe that $u_0 R_{u_0}(\chi_E)\in A_1$. So, taking into account that $v^{q}\in A_{\infty}$, we can choose $r$ very close to $q$ such that  $u_0 R_{u_0}(\chi_E) v^{q-r}\in A_{\infty}$. And since $v_0\in RH_{\infty}$ (which is invariant in $A_{\infty}$, see for instance Lemma 2.1 in \cite{LiOPi}), we have that  $u_0 R_{u_0}(\chi_E) v_0 v^{q-r}\in A_{\infty}$. Now, let

\begin{equation}\label{buv}
C(u, v)=\varphi([u_0 R_{u_0}(\chi_E) v_0 v^{q-r}]_{A_\infty}).
\end{equation}
Using Proposition \ref{qo}, we have that, if $(q/r)' \ge p_0$, 
\begin{eqnarray*}
\int _E f_1^r u v^{q-r}&\lesssim&  C(u, v)\int  \left(\frac {f_2}v\right)^r   R_{u_0}(\chi_E) u v^{q}
\\
&\le& C(u, v) \left|\left| \left(\frac {f_2}v\right)^r\right|\right|_{L^{\frac {q}r, \infty}(uv^q)} ||R_{u_0}(\chi_E)||_{L^{(\frac qr)', 1}(uv^q)}
\\
&\lesssim& C(u, v) \left|\left| \frac {f_2} v\right|\right|_{L^{q, \infty}(uv^q)}^r uv^q(E)^{1/(\frac qr)'},
\end{eqnarray*}
and the first part of the theorem follows using \eqref{kolmo}.
 
Now, we will check the second part of the theorem. In fact, since $v=(Mg)^{-\alpha}$ and $u= u_0 (Mh_1)^{\beta (1-q)}(Mh_2)^{(1-\beta)(1-q)}$,  using Lemma \ref{infinito2} and Propositions \ref{qo} and \ref{later},  we have that 
\begin{eqnarray*}
& & [u_0 R_{u_0}(\chi_E) (Mh_1)^{\beta (1-q)}(Mh_2)^{(1-\beta)(1-q)} v^{q-r}]_{A_\infty}
\\
&\lesssim&  [u_0 R_{u_0}(\chi_E)]^2_{A_{1}}\lesssim K^2_0 \sim [u_0]^{4}_{A_1}.
\end{eqnarray*}

Therefore, we can choose $C(u,v)\sim \phi([u]_{\widehat A_{q,2}})\sim \varphi([u]^{4q}_{\widehat A_{q,2}})$.

\end{proof}

\begin{proof}[Proof of Theorem \ref{Saw}]  By Theorem \ref{extra} and \eqref{mvmv}, we have that 
\begin{eqnarray*}
& &||Sf v^{-1}||_{L^{q, \infty}(uv^q)}\le C(u,v) ||Mf v^{-1}||_{L^{q, \infty}(uv^q)}
\\
&\lesssim&   C(u,v) C_M(u,v) ||f||_{L^{q,1}(u)}:= C_S(u,v) ||f||_{L^{q,1}(u)}. 
\end{eqnarray*}
Now, if $v=(Mh)^{-\alpha}$, we have by \eqref{buv} and \eqref{vava} that there exist two increasing functions $\phi$ and $\Psi$ so that 
$$
C_S(u,v)\lesssim \phi([u]_{\widehat A_{q,2}}) \Psi( [u]_{A_q^{\mathcal R}}, [uv^q]_{A_r^{\mathcal R}}), 
$$
where $r$ is such that   $uv^q \in A_r^{\mathcal R}$. Now, $[u]_{A_q^{\mathcal R}} \lesssim [u]_{\widehat A_{q,2}}$ and  since $uv^q= u_0 \prod_{j=1}^3(Mh_j)^{-\alpha_j}$ for some $h_j\in L^1_{loc}$, $u_0\in A_1$ and $\alpha_j\ge 0$,  we also have that 
\begin{equation}\label{fff}
[uv^q]_{A_r^{\mathcal R}}\lesssim [u_0]_{A_1}^\beta,
\end{equation}
 for some $\beta$ depending on $(\alpha_j)_j$.  Finally, since $u_0$ can be taken so that $[u_0]_{A_1}\le 2 [u]_{\widehat A_{q,2}}^q$, we can conclude that 
$$
C_S(u, v) \lesssim \varphi( [u]_{\widehat A_{q,2}}),
$$
with $\varphi$ increasing and the result follows. 
\end{proof}

\begin{proof}[Proof of Theorem \ref{General}] 

Let $S_i=S_{T_i}$. We will argue directly for $(S_1f_1)(S_2f_2)$ since is a bigger operator than $|(T_1f_1)(T_2f_2)|$. Observe that if $f_i$ are non null bounded and compact support functions it is easy to check that $0<S(f_i)(x)<\infty$ a.e. $x$.  Moreover, it is well know that 
$$
v_{i}:=(S_i(f_i))^{-1}\in A_{\infty}\quad\mbox{and}\quad v_i^\alpha \in A_\infty, \ \forall \alpha>0, 
$$ ($i=1,2$) with {  universal} constants, in the second case only depending on $\alpha$. Then,    following  the strategy in \cite{loptt:loptt, LiOPi, pr:pr}, 
it is enough to prove that
\begin{eqnarray*} \label{eq:eno}
A&:=&\sup_{\lambda>0} \lambda^{p}  (w_{1}^{p/p_1}w_{2}^{p/p_2})\left( \{ \lambda<(S_1f_1(x))(S_2f_2(x))\le 2\lambda \}\right) 
\\
&\lesssim& \left(\|f_1\|_{L^{p_1, 1}(w_1)}\|f_2\|_{L^{p_2, 1}(w_2)}\right)^{p}. 
\end{eqnarray*}
Now, fixed $\lambda>0$,   we observe that 

\begin{eqnarray*}
& &w_{1}^{p/p_1}w_{2}^{p/p_2}\left(\{x\in \mathbb{R}^{n}: \lambda<S_1f_1(x)S_2f_2(x)\le 2\lambda \}\right)
\\
&\lesssim& \lambda^p \int_{\{x\in \mathbb{R}^{n}: \lambda<S_1f_1(x)S_2f_2(x)\le 2\lambda\}} \left(\frac{w_1}{(S_2f_2)^{p_1}}\right)^{\frac p{p_1}}\left(\frac{w_2}{(S_1f_1)^{p_2}}\right)^{\frac p{p_2}}
\\
&\le& \lambda^p \left(\int_{\left\{x\in \mathbb{R}^{n}: \lambda  <\frac{S_1f_1(x)}{v_2(x)}\right\}} w_{1}v_{2}^{p_1}\right)^{p/p_1}\left(\int_{\left\{x\in \mathbb{R}^{n}: \lambda  <\frac{S_2f_2(x)}{v_1(x)}\right\}} w_{2}v_1^{p_2} \right)^{p/p_2}.
\end{eqnarray*}

From here, it follows that
$$
A\lesssim \left|\left| \frac{S_1f_1}{v_2}\right|\right|_{L^{p_1, \infty}(w_{1}v_{2}^{p_1})}^p\left|\left| \frac{S_2f_2}{v_1}\right|\right|_{L^{p_2, \infty}(w_{2}v_{1}^{p_2})}^p.
$$

Since $T_1$ and $T_2$ satisfy condition (C), we have by Theorem \ref{extra} that 
$$
\left|\left| \frac{S_if_i}{v_j}\right|\right|_{L^{p_i, \infty}(w_{i}v_{j}^{p_i})}\lesssim \phi_i([w_i]_{\widehat A_{p_{i},2}}) \left|\left| \frac{Mf_i}{v_j}\right|\right|_{L^{p_i, \infty}(w_{i}v_{j}^{p_i})}, \qquad i,j\in\{1,2\}, \ i\neq j, 
$$
for some increasing functions $\phi_i$. Then,   using \eqref{mvmv},  \eqref{vava} and \eqref{fff}, we have that 
\begin{eqnarray*}
A&\lesssim&  \prod_{i=1}^2 \phi_i([w_i]_{\widehat A_{p_i, 2}}) \prod_{i=1}^2 \Psi_i ([w_i]_{\widehat A_{p_i, 2}} , [w_i v^{p_i}]_{A_r^{\mathcal R}}) \left(\|f_1\|_{L^{p_1, 1}(w_1)}\|f_2\|_{L^{p_2, 1}(w_2)}\right)^{p}
\\ 
&\lesssim& \varphi( [w_1]_{\widehat A_{p_1, 2}}, [w_2]_{\widehat A_{p_2, 2}}) \left(\|f_1\|_{L^{p_1, 1}(w_1)}\|f_2\|_{L^{p_2, 1}(w_2)}\right)^{p}, 
\end{eqnarray*}
with $\varphi$ an increasing function in each variable. 
\end{proof}

The proof of Theorem \ref{Hilbert} follows immediately from Theorem \ref{General} since $H$ satisfies condition (C).

\section{application to bounded variation fourier multipliers}

\begin{proof}[Proof of Corollary \ref{bv1}]
Let $(q_1, q_2; q)$ with $q>1$ and  let $v_j\in \widehat A_{q_j, 2}$ ($j=0,1$). Since 
$L^{q, \infty}(v_1^{q/q_1} v_2^{q/q_2})$  is a Banach space, we have that

\begin{equation*}
||B_m (f_1, f_2)(x)||_{L^{q, \infty}(v_1^{q/q_1} v_2^{q/q_2})} \le \int_{\mathbb R^2} ||H_{t,s} (f_1, f_2)||_{L^{q, \infty}(v_1^{q/q_1} v_2^{q/q_2})}  \, d\mu(t, s). 
\end{equation*}
Now, as mentioned in the introduction
\begin{eqnarray*}
& &H_{t,s} (f_1, f_2)(x)= H_tf_1(x) H_s f_2(x) 
\\
&\approx&\left[ f_1(x) + i e^{2\pi i xt}H(e^{-2\pi i r\cdot} f_1)(x)\right]\left[ f_2(x) + i e^{2\pi i xs}H(e^{-2\pi i r\cdot} f_2)(x)\right], 
\end{eqnarray*}
and thus, since the identity operator and $H$ satisfying condition (C), we can apply Theorem \ref{General} to obtain that, for some increasing function $\phi$, 
$$
||H_{t,s} (f_1, f_2)||_{L^{q, \infty}(v_1^{q/q_1} v_2^{q/q_2})}\lesssim \phi( [v_1]_{\widehat A_{p_1,2}}, [v_1]_{\widehat A_{p_2,2}} ) ||f_1||_{L^{p_1,1}(v_1)}||f_2||_{L^{p_2,1}(v_2)},
$$
and hence, same inequality holds for $B_m$; that is,
$$
||B_m(f_1, f_2)||_{L^{q, \infty}(v_1^{q/q_1} v_2^{q/q_2})}\lesssim \phi( [v_1]_{\widehat A_{p_1,2}}, [v_2]_{\widehat A_{p_2,2}} ) ||f_1||_{L^{p_1,1}(v_1)}||f_2||_{L^{p_2,1}(v_2)}. 
$$
Using  Theorem \ref{th:restrictedGrafakos}, we obtain the result.

\end{proof}

\begin{proof}[Proof of Corollary \ref{bv2}]  By Corollary \ref{bv1} and Remark \ref{mumu}, we have that, for every measurable sets $E$ and $F$, 
$$
|B_m(\chi_E, \chi_F)||_{L^{1/2,\infty}(v_1^{q/q_1}v_2^{q/q_2})}\lesssim v_1(E) v_2(F). 
$$
From here, using the same argument than in the proof of Corollary 1.4,  in \cite{bcls} (see also \cite{cgs:cgs}) we obtain the result.

\end{proof}

\end{document}